\documentclass[letterpaper, 10 pt, conference]{ieeeconf}
\IEEEoverridecommandlockouts
\usepackage{hyperref}

\usepackage{amsmath,amsthm,amsfonts,amssymb}
\usepackage[ruled,vlined]{algorithm2e}
\usepackage{xcolor}
\usepackage{enumerate}
\usepackage{cite}
\usepackage{csquotes}

\usepackage{graphicx}
\DeclareGraphicsExtensions{.eps}
\usepackage{subcaption}

\usepackage{multicol}
\usepackage{multirow}

\usepackage{calc}
\usepackage{tikz}
\usetikzlibrary{shapes, decorations.markings, calc}

\theoremstyle{thmstyle}
\newtheorem{Theorem}{Theorem}
\newtheorem{Proposition}[Theorem]{Proposition}

\newtheorem{Assumption}{Assumption}
\theoremstyle{remark}
\newtheorem{Remark}{Remark}

\def \ba{\begin{array}}
\def \ea{\end{array}}

\def \bea{\begin{eqnarray}}
\def \eea{\end{eqnarray}}

\def \be{\begin{equation}}
\def \ee{\end{equation}}

\def \BEA{\begin{eqnarray*}}
\def \EEA{\end{eqnarray*}}

\def \BE{\begin{equation*}}
\def \EE{\end{equation*}}

\def \+{\dagger}

\def \bb{\mathbb}

\def \sf{\mathsf}

\def \disp{\displaystyle}

\def \tcol{\textcolor}

\def \S{\mathtt{S}}
\def \I{\mathtt{I}}
\def \D{\mathtt{D}}
\def \R{\mathtt{R}}

\def \U{\mathtt{U}}
\def \T{\mathtt{T}}

\definecolor{blue1}{RGB}{128, 191, 255}
\definecolor{redd1}{RGB}{255, 153, 128}
\definecolor{yelw1}{RGB}{255, 179, 26}
\definecolor{prpl1}{RGB}{204, 153, 255}
\definecolor{gren1}{RGB}{153, 230, 153}
\definecolor{blue2}{RGB}{51   153  255}
\definecolor{redd2}{RGB}{255  71   26}
\definecolor{yelw2}{RGB}{255  204  0}
\definecolor{prpl2}{RGB}{153  51   255}
\definecolor{gren2}{RGB}{51   204  51}
\definecolor{mygreen}{rgb}{0,0.7,0}

\begin{document}

\title{Effective Testing Policies for Controlling an Epidemic Outbreak}

\author{Muhammad Umar B. Niazi \quad Alain Kibangou \quad Carlos Canudas-de-Wit \\ Denis Nikitin \quad Liudmila Tumash \quad Pierre-Alexandre Bliman
\thanks{M.U.B. Niazi, C. Canudas-de-Wit, A. Kibangou, D. Nikitin, and L. Tumash are with Universit\'e Grenoble Alpes, Inria, CNRS, Grenoble INP, GIPSA-Lab, 38000 Grenoble, France.
Emails in the respective order: \href{mailto:muhammad-umar-b.niazi@gipsa-lab.fr}{muhammad-umar-b.niazi@gipsa-lab.fr}, \href{mailto:carlos.canudas-de-wit@gipsa-lab.fr}{carlos.canudas-de-wit@gipsa-lab.fr}, \href{mailto:alain.kibangou@univ-grenoble-alpes.fr}{alain.kibangou@univ-grenoble-alpes.fr}, \href{mailto:denis.nikitin@gipsa-lab.fr}{denis.nikitin@gipsa-lab.fr}, \href{mailto:liudmila.tumash@gipsa-lab.fr}{liudmila.tumash@gipsa-lab.fr}}
\thanks{P.-A. Bliman is with Sorbonne Universit\'{e}, Universit\'{e} Paris-Diderot SPC, Inria, CNRS, Laboratoire Jacques-Louis Lions, \'{e}quipe Mamba, 75005 Paris, France.
Email: \href{mailto:pierre-alexandre.bliman@inria.fr}{pierre-alexandre.bliman@inria.fr}}
\thanks{
This work is partially supported by European Research Council (ERC) under the European Union's Horizon 2020 research and innovation programme (ERCAdG no. 694209, Scale-FreeBack, website: \url{http://scale-freeback.eu/}) and by Inria, France, in the framework of the Inria Mission Covid-19 (Project HealthyMobility).}
}

\maketitle

\begin{abstract}
Testing is a crucial control mechanism for an epidemic outbreak because it enables the health authority to detect and isolate the infected cases, thereby limiting the disease transmission to susceptible people, when no effective treatment or vaccine is available. In this paper, an epidemic model that incorporates the testing rate as a control input is presented. The proposed model distinguishes between the undetected infected and the detected infected cases with the latter assumed to be isolated from the disease spreading process in the population. Two testing policies, effective during the onset of an epidemic when no treatment or vaccine is available, are devised: (i) best-effort strategy for testing (BEST) and (ii) constant optimal strategy for testing (COST). The BEST is a suppression policy that provides a lower bound on the testing rate to stop the growth of the epidemic. The COST is a mitigation policy that minimizes the peak of the epidemic by providing a constant, optimal allocation of tests in a certain time interval when the total stockpile of tests is limited. Both testing policies are evaluated by their impact on the number of active intensive care unit (ICU) cases and the cumulative number of deaths due to COVID-19 in France.
\end{abstract}

\section{Introduction}
On March 16, 2020, five days after the World Health Organization (WHO) declared COVID-19 to be a pandemic, the Director-General of WHO, Dr. T. A. Ghebreyesus, ended his media briefing with a simple message to all countries: ``Test, test, test". This is because testing is well-known to be a crucial control mechanism for epidemics when no effective treatment or vaccine is available \cite{chowell2003}. It limits the spread of disease to the susceptible population by enabling the health authority to detect and isolate the infected people.

Recently, several models and techniques have been developed to study testing policies for epidemic control. Similar to a resource allocation problem in epidemic control \cite{nowzari2017}, the optimal test allocation problem is posed in \cite{pezzutto2020} as a well-known sensor selection problem in control theory. A similar problem studied in \cite{ely2020} considers the specificity and sensitivity of tests and solves the optimal test allocation problem as a welfare maximization problem. However, the underlying assumption in these papers is the availability of information portfolios of all the people in a society, thus limiting their applicability to larger scales.
Another line of research studies testing policies from an economic perspective, for instance, \cite{piguillem2020} and \cite{berger2020} aim to find an optimal trade-off between the economic activity and epidemic mitigation. In \cite{charpentier2020}, an optimal trade-off between testing effort and lockdown intervention is computed numerically under the constraint of limited Intensive Care Units (ICU). However, none of these papers provide a control-theoretic perspective of testing, namely how testing can be used to change the dynamics of an epidemic?

In this paper, we study testing policies in terms of testing rate, i.e., number of tests performed per day, with the aim to suppress and/or mitigate an epidemic outbreak. We introduce a new model with the acronym SIDUR---susceptible (S), undetected infected (I), detected infected (D), unidentified recovered (U), and identified removed (R)---that has the testing rate as a control input. The influence of the control input in the SIDUR model is directly linked with the testing specificity, which determines the probability of detecting an infected person by a test. 
Similar to \cite{giordano2020,liu2020,ducrot2020}, we distinguish the undetected infected from the detected infected population. Only the undetected infected people participate in the disease transmission, and the detected infected are assumed to be either quarantined and/or hospitalized. Moreover, the unidentified recovered people, who were infected and then recovered naturally without getting detected, are distinguished from the identified removed people, who were infected and then recovered or died in a hospital or quarantine after getting detected.

We propose two testing policies for epidemic control: (1)~Best-effort strategy for testing (BEST) and (2)~Constant optimal strategy for testing (COST). 
The BEST is a suppression policy for an epidemic that computes the minimum testing rate to immediately stop the growth of the epidemic when applied from a certain day onward. The COST is a mitigation policy for an epidemic that reduces, but not necessarily stops, the growth of the epidemic and is meaningful only when the total stockpile of tests is limited. It provides the constant optimal testing rate that must be utilized until the day when the stockpile of tests finishes. We evaluate both testing policies by predicting their impact on the number of active intensive care unit (ICU) cases and the cumulative number of deaths due to COVID-19 in France. Both testing policies are effective at the beginning of an epidemic outbreak when the treatments or vaccines against the disease are not available.

This paper focuses solely on the SIDUR model design and testing policies, and excludes model estimation and validation, which is reported in the detailed version \cite{niazi2020}. The paper is organized as follows: Section~\ref{sec:model} presents the SIDUR model, identifies the measured outputs in relation to the model states, and provides a partial solution of the model. Then, in Section~\ref{sec:BEST} and \ref{sec:COST}, we propose the BEST and COST policies, and evaluate them for the case of COVID-19 outbreak in France from February to July 2020. Finally, Section~\ref{sec:conclusion} presents the concluding remarks.

\section{Model Design} \label{sec:model}

\subsection{Model Equations}
We introduce a five-compartment model, SIDUR, with the purpose of devising and evaluating testing policies. We consider the timescale of days and, at time $t\in\bb{R}_{\geq 0}$, each compartment is characterized by a single state representing its population, where
\begin{itemize}
    \item $x_\S(t)$: Number of {\em susceptible} people who are prone to get infected by the disease. 
    \item $x_\I(t)$: Number of {\em undetected infected} people who are infected but not diagnosed by the health authority.
    \item $x_\D(t)$: Number of {\em detected infected} people who are infected and are diagnosed through a test.
    \item $x_\U(t)$: Number of {\em unidentified recovered} people who were infected and then recovered naturally without getting diagnosed.
    \item $x_\R(t)$: Number of {\em identified removed} people who were infected, diagnosed, and then either recovered or died.
\end{itemize}

\begin{Assumption} \label{assumption}
\begin{enumerate}[(i)]
\item The total population $N$ remains constant during the evolution of the epidemic
\[
x_\S(t) + x_\I(t) + x_\D(t) + x_\U(t) + x_\R(t) = N, \quad \forall\,t\in\bb{R}_{\geq 0}.
\]
\item Only the undetected infected $x_\I(t)$ are responsible for the disease transmission to the susceptible people $x_\S(t)$.
\item All the deaths occurring due to the epidemic are known, reported, and included in the identified removed $x_\R(t)$.
\item After recovery, people acquire permanent immunity and are not infected again.
\end{enumerate}
\end{Assumption}

Let $\beta\geq 0$ be the {\em infection rate} defined as the product of the contact rate by the probability of disease transmission during a contact between a susceptible and an infected persons. The value of $\beta$ not only depends on the disease biology but is also influenced by the implementation of non-pharmaceutical interventions (NPI) such as social distancing, lockdown, travel restrictions, and preventive measures. Thus, $\beta=\beta(t)$ is a time-varying parameter that depends on the extent to which the NPIs are implemented. Depending on the time periods when different NPI policies are implemented, we assume $\beta$ to be piece-wise constant.

Let $x_\T(t)$ be the {\em testable population} defined as
\[
\ba{ccl}
x_\T(t) &:=& x_\I(t) + (1-\theta) \left(x_\S(t) + x_\U(t)\right) 
\ea
\]
where $\theta\in[0,1]$ is the {\em testing specificity parameter}. Note that the infected population $x_\I(t)\leq x_\T(t)$, for every $t$, where the equality holds if $\theta=1$. Given $\theta\in[0,1]$, the probability of detecting an infected person per test in a homogeneous population structure is $x_\I(t)/x_\T(t)$. Thus, if $\theta = 1$, the probability of detection is equal to one, and if $\theta = 0$, the probability of detection is the lowest. Depending on the testing policy, testing specificity, and the level of contact tracing, we consider $\theta=\theta(t)$ to be time-varying and piece-wise constant.

Finally, let $\gamma>0$ and $\rho>0$ be the {\em recovery rate} and {\em removal rate}, respectively, which are assumed to be constants that depend on the disease biology. The inverse of the recovery rate $1/\gamma$ is defined to be the average recovery time after which a typical undetected infected person naturally recovers, whereas the inverse of the removal rate $1/\rho$ is defined to be the average removal time after which a typical detected infected person recovers or dies.

The SIDUR model, depicted in Fig.~\ref{fig:SIDUR_block}, is given by
\begin{subequations}
	\label{eq401}
	\begin{eqnarray}
	\label{eq401a}
	\dot{x}_\S(t) &=& \disp - \beta \; x_\S(t) \frac{x_\I(t)}{N} \\ [0.25em]
	\label{eq401b}
	\dot{x}_\I(t) &=& \disp \beta \; x_\S(t) \frac{x_\I(t)}{N} - u(t) \frac{x_\I(t)}{x_\T(t)} -\gamma x_\I(t) \\ 
	\label{eq401c}
	\dot{x}_\D(t) &=& u(t) \frac{x_\I(t)}{x_\T(t)} - \rho x_\D(t) \\ [0.5em]
	\label{eq401d}
	\dot{x}_\U(t) &=&\gamma x_\I(t)\\ [1em]
	\label{eq401e}
	\dot{x}_\R(t) &=&\rho x_\D(t)
	\end{eqnarray}
\end{subequations}
where $u(t)\geq 0$ is the testing rate. Notice that the model \eqref{eq401} fulfils Assumption~\ref{assumption} for any $u(t)\geq 0$.

\begin{figure}[t]
\begin{center}
\begin{tikzpicture}[scale=0.6]

    
	\draw[blue1, fill=blue1, very thick, rounded corners, opacity = 0.5] (-1.5,2.75) rectangle (1.5,4.25);
	\node at (0,3.5) {\tcol{redd2}{\bf S}usceptible};
		
	\draw[-latex, very thick] (1.52,3.5) -- (3.48,3.5);
	\node[text width=2.5cm, anchor=south, align=center] at (2.5,3.75) {\small Infection \\ [-1.25mm] rate};
	
	\draw[redd1, fill=redd1, very thick, rounded corners, opacity = 0.5] (3.5,2.75) rectangle (6.5,4.25);
	\node[text width=2cm, align=center] at (5,3.5) {Undetected \tcol{redd2}{\bf I}nfected};
		
	\draw[-latex, dotted, very thick] (5,2.73) -- (5,1.27);
	\node[text width=2.5cm, anchor=east, align=center] at (6.25,2) {\small Testing \\ [-1.25mm] rate};
	
	\draw[redd1, fill=redd1, very thick, rounded corners, opacity = 0.5] (3.5,-0.25) rectangle (6.5,1.25);
	\node[text width=2cm, align=center] at (5,0.5) {\tcol{redd2}{\bf D}etected Infected};

	\draw[-latex, very thick] (6.52,3.5) -- (8.48,3.5);
	\node[text width=2.5cm, anchor=south, align=center] at (7.5,3.75) {\small Recovery \\ [-1.25mm] rate};
	
	\draw[gren1, fill=gren1, very thick, rounded corners, opacity = 0.5] (8.5,2.75) rectangle (11.5,4.25);
	\node[text width=2cm, align=center] at (10,3.5) {\tcol{redd2}{\bf U}nidentified recovered};
	
	\draw[-latex, very thick] (6.52,0.5) -- (8.48,0.5);
	\node[text width=2.5cm, anchor=south, align=center] at (7.5,0.75) {\small Removal \\ [-1.25mm] rate};
	
	\draw[gren1, fill=gren1, very thick, rounded corners, opacity = 0.5] (8.5,-0.25) rectangle (11.5,1.25);
	\node[text width=2cm, align=center] at (10,0.5) {Identified \tcol{redd2}{\bf R}emoved};
\end{tikzpicture}
\caption{Block diagram of SIDUR model.}
\label{fig:SIDUR_block}
\end{center}
\end{figure}
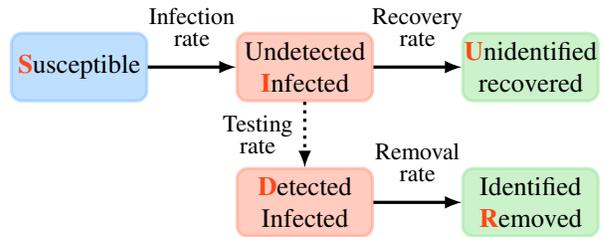

\subsection{Control Input and Measured Outputs}
The testing rate $u(t)$, which is defined to be the number of tests performed per day, is the control input in SIDUR model. 
In order to diagnose the infected people at time $t$, the tests are allocated to a proportion of the testable population $x_\T(t)$. The testing specificity parameter $\theta$ allows for the adjustment of the testable population to accommodate for the detection rate of tests. 

The data reported  by the health authorities are chosen as the measured outputs of the SIDUR model. They are five measurements, related directly to the states as follows:
\begin{itemize}
    \item Cumulative number of detected people
    \begin{equation}\label{total_infected_cases}
        y_1(t)= x_\D(t) + x_\R(t).
    \end{equation}
    \item Cumulative number of removed people
    \begin{equation}\label{total_recovered_cases}
        y_2(t)= x_\R(t).
    \end{equation}
    \item Number of positively tested people per day
    \begin{equation}\label{testedpos_cases}
        y_3(t)= u(t) \frac{x_\I(t)}{x_\T(t)}.
    \end{equation}
\end{itemize}

Let the number of active ICU patients (or ICU beds occupied) be denoted as $B(t)$ and the cumulative number of deaths (or extinct cases) be denoted $E(t)$. Note that $B(t)$ and $E(t)$ are also measured and reported, and are functions of the number of active infected cases and the cumulative number of infected cases, respectively.
\begin{itemize}
    \item Number of active ICU cases
    \be \label{activeICU_cases}
    y_4(t) := B(t) = g(A(t-\psi^{-1}))
    \ee
    where $A(t) = x_\I(t)+x_\D(t)$ is the number of active infected cases, $\psi^{-1}$ is the average time period which a typical infected ICU case takes from getting infected to being admitted to the ICU, and the function $g$ is to be chosen such that $y_4(t) = B(t)$ fits the data.
    \item Cumulative number of deaths
    \be \label{total_dead_cases}
    y_5(t) := E(t) = h(I(t-\phi^{-1}))
    \ee
    where $I(t)=N-x_\S(t)$ is the cumulative number of infected cases, $\phi^{-1}$ is the average time period which a typical non-surviving case takes from getting infected to death, and the function $h$ is to be chosen such that $y_5(t)=E(t)$ fits the data.
\end{itemize}

\subsection{Reproduction Number}
An important quantity to determine the epidemic potential at the onset of a disease is the basic reproduction number $R_0$. In the later stages, however, the effective reproduction number $R(t)$ is more suitable \cite{hethcote2000,rothman2008} to predict the spread of the disease.
The $R(t)$ for the SIDUR model is defined as the ratio of the number of newly infected people over the number of newly detected and recovered people at time $t$, i.e.,
\[
R(t) = \frac{\beta(t) x_\S(t) \frac{x_\I(t)}{N}}{u(t) \frac{x_\I(t)}{x_\T(t)} + \gamma x_\I(t)} = \frac{\beta(t)}{\frac{u(t)}{x_\T(t)} + \gamma} \frac{x_\S(t)}{N}.
\]
If $R(t)<1$, then more people are being detected and/or recovering than the people being infected at time $t$, implying that $x_\I(t)$ will decrease in the near future. If $R(t)>1$, then more people are being infected than the people being detected and recovering at time $t$, implying that $x_\I(t)$ will increase in the near future. The basic reproduction number $R_0=R(0)$ and is given by
\[
R_0 \approx \frac{\beta(0)}{\frac{u(0)}{(1-\theta) N} + \gamma}
\]
where we assume that $x_\S(0)\approx N$ and $x_\T(0)\approx (1-\theta)N$.
This expression of $R_0$ can also be obtained by following the methodology of \cite{driessche2008}.

\subsection{Partial Solution of SIDUR Model}

Consider the {\em infection time} $\xi$ defined as $d\xi = x_\I(t) dt$, where
\be \label{eq:xi}
\xi(t) = \int_0^t x_\I(\eta) d\eta
\ee
and
\be \label{eq:t(xi)}
t(\xi) = \int_0^\xi \frac{1}{x_\I(\eta)} d\eta
\ee
with $\xi = 0$ if and only if $x_\I \equiv 0$ on $[0,t]$.
Moreover,
\[
\frac{dx_\I}{dt} (t) = 0 \quad \Leftrightarrow \quad \frac{dx_\I}{d\xi} (\xi) = 0
\]
because $dx_\I/dt = (dx_\I/d\xi)(d\xi/dt) = (dx_\I/d\xi) x_\I$.
Thus, at the time $t_p$ of the infection peak, we have
\[
\frac{dx_\I}{dt}\bigg|_{t=t_p} = 0 \quad \Leftrightarrow \quad \frac{dx_\I}{d\xi}\bigg|_{\xi=\xi_p} = 0
\]
where
\[
\xi_p = \int_0^{t_p} x_\I(\eta) d\eta.
\]

\begin{Proposition}
Consider the infection time $\xi$, then:
\be \label{soln:xS}
x_\S(\xi) = x_\S^0 e^{-\frac{\beta}{N}\xi}
\ee
\be \label{soln:xU}
x_\U(\xi) = x_\U^0 + \gamma \xi
\ee
where $x_\S^0 = x_\S(0)$ and $x_\U^0=x_\U(0)$.
\end{Proposition}
\begin{proof}
First, from \eqref{eq401a}, we obtain
\[
\frac{dx_\S}{d\xi} = -\frac{\beta}{N} x_\S(\xi)
\]
whose solution is given by \eqref{soln:xS}. Then, from \eqref{eq401d}, we obtain
\[
\frac{dx_\U}{d\xi} = \gamma
\]
whose solution is given by \eqref{soln:xU}.
\end{proof}


The testing specificity parameter $\theta$ can be close to one under different circumstances. For instance, if the testing is assisted with efficient contact tracing, then $\theta\approx 1$. However, if the tests are performed only on people with apparent symptoms of the disease, then $\theta\approx 1$ also. If we assume a scenario unlike those mentioned above, then $\theta \ll 1$. In addition, if we assume that $x_\S(t)+x_\U(t) \approx N$, then we can obtain an approximate solution of $x_\I(\xi)$ by assuming $x_\T(t) \approx (1-\theta)N$.

\begin{Proposition}
Assume the testable population $x_\T(t) \approx (1-\theta) N$.
Then, an approximate solution of the undetected infected population with respect to the infection time $\xi$ is
\be \label{soln:xI}
x_\I(\xi) = \disp x_\I^0 + x_\S^0 ( 1- e^{-\frac{\beta}{N}\xi}) - \gamma \xi  - \int_0^\xi \frac{u(\eta)}{(1-\theta)N} d\eta.
\ee
where $x_\I^0 = x_\I(0)$.
\end{Proposition}
\begin{proof}
From \eqref{eq401b} and \eqref{soln:xS}, we have
\[
\frac{dx_\I}{d\xi} = \frac{\beta}{N} x_\S^0 e^{-\frac{\beta}{N}\xi} - \gamma - \frac{u}{(1-\theta)N}
\]
where we assumed $x_\T = (1-\theta)N$. Integrating both sides from $0$ to $\xi$ gives \eqref{soln:xI}.
\end{proof}

\section{Best-Effort Strategy for Testing} \label{sec:BEST}

Suppose that the tests can be produced and supplied easily, and there is no upper bound on the total stockpile of tests. Moreover, we say that, at any time $t$, the epidemic is {\em spreading} if $x_\I(t)$ is increasing, i.e., $R(t)>1$. On the other hand, we say that the epidemic is {\em non-spreading} if $x_\I(t)$ is not increasing, i.e., $R(t) \leq 1$. Then, the {\it best effort strategy for testing} (BEST) at any given time $t^*$ is the minimum number of tests to be performed per day from time $t^*$ onward such that the epidemic switches from spreading to non-spreading at $t^*$. In other words, the BEST policy provides the smallest lower bound on the testing rate that is sufficient to change, at a given time $t^*$, the course of the epidemic from spreading to non-spreading. We remark that the BEST policy is meaningful only when the epidemic is spreading.

\subsection{Computation of the BEST policy}

Define a function
\[
c^*(t)=x_\T(t)\left|\frac{ \beta(t)}{N}x_\S(t)-\gamma\right|_+
\]
where, by definition, for any scalar $z$, $|z|_+=z$ if $z > 0$ and $|z|_+=0$ otherwise. 
\begin{Proposition} \label{prop:BEST}
Assume that the infection rate $\beta$ is non-increasing while the testing specificity parameter $\theta$ is non-decreasing on a time interval $[t^*,t_1)$, for some $t^* < t_1$. Then, the BEST policy at time $t^*$ is given by 
\[
u(t)=c^*(t^*)=x_\T(t^*)\left|\frac{ \beta(t^*)}{N}x_\S(t^*)-\gamma\right|_+,\quad \forall t\in[t^*,t_1).
\]
\end{Proposition}
\begin{proof}
\mbox{}
In order to prove that $c^*(t^*)$ for $t\in [t^*,t_1)$ is the BEST at $t^*$,  we show the following:
\begin{enumerate}[(i)]
    \item If $u(t) > c^*(t)$ (resp., $u(t) \geq c^*(t)$) for any $t\in [t^*,t_1)$, then $x_\I$ is decreasing (resp., non-increasing) on $[t^*,t_1)$.
    \item If $u(t) > c^*(t^*)$ (resp., $u(t) \geq c^*(t^*)$) for any $t\in [t^*,t_1)$, then $x_\I$ is decreasing (resp., non-increasing) on $[t^*,t_1)$.
\end{enumerate}

Assume that $u(t) > c^*(t)$ on $[t^*,t_1)$.
Then, 
$
\Phi(t) := \beta(t) \frac{x_\S(t)}{N} -\frac{u(t)}{x_\T(t)} - \gamma < 0
$
which implies that $x_\I$ is decreasing because $\dot x_\I(t) = \Phi(t) x_\I(t)$ almost everywhere. 
If only the weaker assumption $u(t) \geq c^*(t)$ on $[t^*,t_1)$ is fulfilled, then, by using the continuity of the solutions of ODE with respect to perturbations of the right-hand side, one gets that $x_\I$ is non-increasing.

Assume now that $u(t) > c^*(t^*)$ on $[t^*,t_1)$, where $c^*(t^*)$ is constant. Then, by continuity, $u(t) > c^*(t)$ on a certain interval $[t^*,t_2)$, for some $t_2 \in (t^*,t_1]$. As a consequence of the result (i) shown previously, $x_\I$ decreases on $[t^*,t_2)$. Moreover, assume that $t_2$ is the maximal point in $(t^*,t_1]$ having this property. In order to show that $t_2=t_1$, it is sufficient to show that $u(t_2)> c^*(t_2)$, otherwise one may consider a larger value for $t_2$ which will lead to a contradiction with the fact that it is maximal. Since $x_\I$ decreases on $[t^*,t_2)$ and $\theta$ is non decreasing, we can conclude that $x_\T$ also decreases on this interval. On the other hand, since $x_\S$ is always decreasing and $\beta$ is non increasing, one can conclude that $c^*(t)$ also decreases on $[t^*,t_2)$. This is obtained by upper bounding $c^*(t)$. Thus, one has $c^*(t^*) > c^*(t)$, which implies that that $u(t_2)>c^*(t_2)$.
Therefore, as $t_2=t_1$, we have established that $x_\I$ decreases on the whole interval $[t^*,t_1)$.
For the case where $u(t) \geq c^*(t^*)$, we can use the same argument of continuity of the trajectories.

From the previous results, one deduces that the BEST policy is given by $c^*(t^*)$ and the testing rate $u(t) \geq c^*(t^*)$ for $t\in[t^*,t_1)$.
If $u(t)<c^*(t^*)$, for $t\in[t^*,t_1)$, then one can show that the epidemic goes on spreading in the interval $[t^*,t_1)$. Hence, $u(t)=c^*(t^*)$ is the BEST policy at $t^*$. 
\end{proof}

\begin{Remark}
Requiring that $\beta$ must not increase and $\theta$ must not decrease in the interval $[t^*,t_1)$ for some $t_1>t^*$ is necessary for the BEST policy. 
It is thus important to keep the external conditions that determine the values of $\beta$ and $\theta$ either constant or such that $\beta$ decreases (e.g., through the implementation of lockdown) and/or $\theta$ increases (e.g., through efficient contact tracing). On the other hand, the case where $\beta$ decreases and/or $\theta$ increases at some time $t_1>t^*$ has the effect of speeding up the suppression of the epidemic under BEST policy. \hfill $\lrcorner$
\end{Remark}

\subsection{BEST Policy for the COVID-19 Case of France}

Given the COVID-19 data of France, we first estimate the parameters of the SIDUR model by fitting the model outputs \eqref{total_infected_cases}, \eqref{total_recovered_cases}, and \eqref{testedpos_cases} with the corresponding data (see \cite{niazi2020} for details). Then, we compute $c^*(t^*)$ for different values of $t^*$ from January~24 to March~15. Figure~\ref{fig:with-BEST-2} shows the testing rate $u(t^*)$ required by the BEST policy if it is implemented from day $k=\lfloor t^* \rfloor$, where $\lfloor\cdot\rfloor$ denotes the floor function, and the corresponding value of the peak of undetected infected $x_\I(t)$. It can be seen that the required testing rate increases exponentially if the BEST policy is not implemented early on in the epidemic phase. This is because the number of undetected infected $x_\I(t)$ becomes very large, which increases the difficulty of controlling the epidemic.
 
\begin{figure}[!t]
    \centering
    \includegraphics[width=0.4\textwidth]{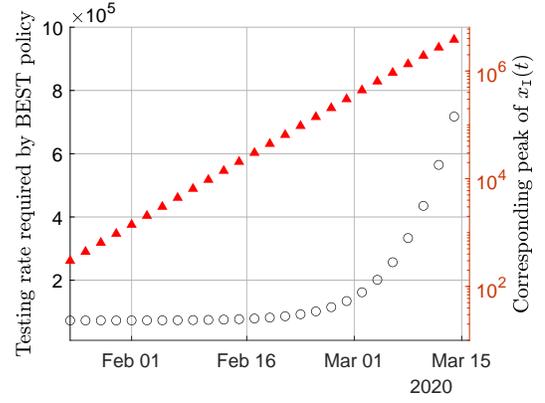}
    \caption{Testing rate $u(t)$ required by the BEST policy (left y-axis, black) and the predicted value of the corresponding epidemic peak (right y-axis, red, in logscale) with respect to time.}
    \label{fig:with-BEST-2}
\end{figure}
\begin{figure}[!t]
    \centering
    \includegraphics[width=0.4\textwidth]{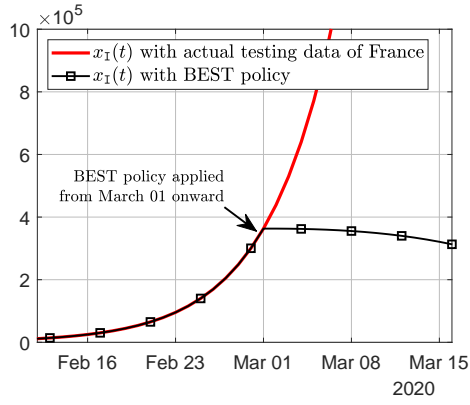}
    \caption{Predicted number of infected cases $x_\I(t)$: actual testing scenario vs. BEST}
    \label{fig:with-BEST}
\end{figure}

Let us now consider a scenario where BEST is implemented on March~01, 2020. Then, Figure~\ref{fig:with-BEST} depicts the evolution of undetected infected $x_\I(t)$ when $u(t)$ is given by the data on the number of tests per day in France and when $u(t)$ is given by BEST policy. To evaluate BEST, we use $u(t)$ as given by the data until February~29 and then by the BEST policy from March~01 onward. In the actual COVID-19 case of France, our model predicted the peak of the undetected infected $x_\I(t)$ to be about $6$ million. However, if the BEST policy was implemented, the peak of $x_\I(t)$ could have been reduced to $363,169$ with the testing rate $c^*\approx 147,000$, which is quite feasible. The impact of the BEST policy on the number of active ICU cases $B(t)$ and the cumulative number of deaths $E(t)$ is evaluated after fitting the model outputs \eqref{activeICU_cases} and \eqref{total_dead_cases} with the data in the actual testing scenario. The results are illustrated in Figure~\ref{fig:ICU-BEST} and \ref{fig:Death-BEST}.
From the figures, we observe that the peak of the number of active ICU patients could have been reduced by $34.71\%$ and the number of deaths could have been reduced by $74.45\%$ if the BEST policy was implemented.

\begin{figure}[!t]
    \centering
    \includegraphics[width=0.4\textwidth]{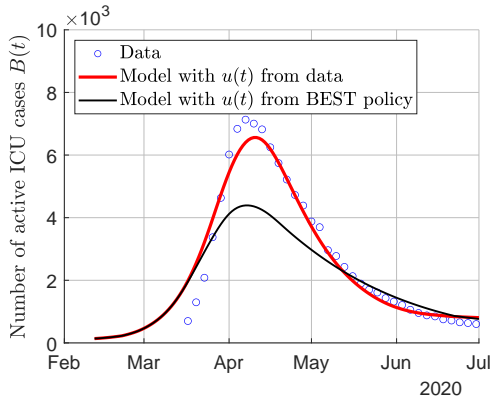}
    \caption{The prediction of the number of active ICU cases $B(t)$: actual scenario vs. BEST policy.}
    \label{fig:ICU-BEST}
\end{figure}
\begin{figure}[!t]
    \centering
    \includegraphics[width=0.4\textwidth]{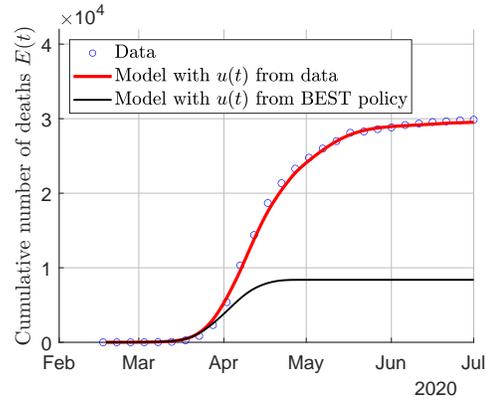}
    \caption{The prediction of the cumulative number of deaths $E(t)$: actual scenario vs. BEST policy.}
    \label{fig:Death-BEST}
\end{figure}

\section{Constant Optimal Strategy for Testing} \label{sec:COST}

Suppose that the total stockpile of tests is limited and given by $r_{\max}\in\bb{N}$, where $\bb{N}$ denotes the set of natural numbers. Then, the problem is to find the optimal testing rate that minimizes the epidemic peak, i.e., the peak value of the undetected infected population $x_\I(t)$. We seek an optimal trade-off because if the testing rate is chosen high, then the stockpile finishes early and leaves us with no tests at hand, which could result in a second wave of the epidemic more intense than the first wave. On the other hand, if the testing rate is chosen low, then it is impossible to limit the epidemic spread (or `flatten the curve'), which could challenge the available medical facilities of the country. 

In this section, we determine a constant optimal daily allocation of a limited stockpile of tests $r_{\max}$ minimizing the epidemic peak. Such an allocation is called the {\em constant optimal strategy for testing} (COST). Given the total stockpile of tests $r_{\max}$,  the constant testing rate is given by
\begin{equation} \label{u-COST}
    u(t) = \left\{ \begin{array}{ll}
          C,    & \text{if}~ 0 \leq t \leq  T \\
          0, & \text{if}~ t>T
    \end{array} \right.
\end{equation}
where the time period $T:=r_{\max} / C$ represent the duration of the testing policy once $C$ is determined. Therefore, the approximate solution \eqref{soln:xI} of undetected infected population is given by
\begin{equation}\label{soln:xI_COST}
x_\I(\xi) = \left\{
\arraycolsep=3pt\ba{ll}
\disp x_\I^0 + x_\S^0 \left(1 - e^{-\frac{\beta}{N}\xi} \right) - \frac{ C \xi}{(1-\theta)N} - \gamma \xi, & \xi < \Xi \\ [0.75em]
\disp x_\I^0 + x_\S^0 \left(1 - e^{-\frac{\beta}{N}\xi} \right) - \frac{ C\Xi}{(1-\theta)N} - \gamma \xi, & \xi \ge \Xi
\ea \right.
\end{equation}
where $\Xi$ denotes the moment in the infection time coordinate when the stockpile of tests finishes. This $\Xi$ is given by an implicit formula
\begin{equation}\label{xi_star_def}
\int\limits_0^{\Xi} \frac{d\xi}{x_\I^0 + x_\S^0 \left(1 - e^{-\frac{\beta}{N}\xi} \right) - \frac{ C \xi}{(1-\theta)N} - \gamma \xi} = T
\end{equation}
which is derived from \eqref{eq:t(xi)} as
\[
\int\limits_0^{\Xi} \frac{d\xi}{x_\I(\xi)} = T
\]
where $T:= r_{\max}/C$ and $x_\I(\xi)$ is given by \eqref{soln:xI} and \eqref{soln:xI_COST}.

\subsection{Computation of the COST policy}

The epidemic spread can be controlled through a suitable testing rate $u(t)=C$, for $t\in[0,T]$. However, when the limited stockpile of tests finishes, the disease may start spreading again, resulting in a second wave of the epidemic. Let the {\it testing phase} be the time interval when the tests are being performed, i.e., $t\in[0,T]$, and let the {\it non-testing phase} be the time after the stockpile finishes, i.e., $t>T$. Thus, it is reasonable to suppose that there can be two epidemic peaks --- one that arrives during the testing phase and the other that arrives during the non-testing phase.

Let $\xi_{p_1}$ and $\xi_{p_2}$ be the moments in infection time coordinate when the two epidemic peaks occur. We will prove that $0 < \xi_{p_1} \leq \Xi \leq \xi_{p_2}$ in the following:. 
If the two epidemic peaks exist, then
\[
\begin{aligned}
\left.\frac{dx_\I}{d\xi}\right|_{\xi=\xi_{p_1}} &= x_\S^0 \frac{\beta}{N} e^{-\frac{\beta}{N}\xi_{p_1}} - \frac{C}{(1-\theta)N} - \gamma = 0 \\
\left.\frac{dx_\I}{d\xi}\right|_{\xi=\xi_{p_2}} &= x_\S^0 \frac{\beta}{N} e^{-\frac{\beta}{N}\xi_{p_2}} - \gamma = 0
\end{aligned}
\]
where the peak times are given as
\[
\xi_{p_1} = \frac{N}{\beta} \ln R_1, \qquad
\xi_{p_2} = \frac{N}{\beta} \ln R_2
\]
with the reproduction numbers
\[
R_1 = \frac{x_\S^0 \beta}{\frac{C}{1-\theta} + \gamma N}, \quad
R_2 = \frac{x_\S^0 \beta}{\gamma N}, \quad R_2>R_1>1.
\]
Therefore, we have $0< \xi_{p_1} < \xi_{p_2}$. Further, notice that the first peak occurs only if $\xi_{p_1} < \Xi$, otherwise it is ill-defined. Similarly, $\xi_{p_2} > \Xi$.

The values of two epidemic peaks are
\begin{equation}\label{x_I_two_peaks}
\begin{aligned}
x_\I^{p_1} &= x_\I^0 + x_\S^0 \left(1 - \frac{1}{R_1} \right) - x_\S^0 R_1 \ln R_1 \\
x_\I^{p_2} &= x_\I^0 + x_\S^0 \left(1 - \frac{1}{R_2} \right) - x_\S^0 R_2 \ln R_2 - \frac{C \Xi}{(1-\theta)N}
\end{aligned}
\end{equation}
where $x_\I^{p_1} = x_\I(\xi_{p_1})$ and $x_\I^{p_2} = x_\I(\xi_{p_2})$.
It is obvious that $x_\I^{p_1}$ decreases as $C$ increases, because $x_\I^{p_1}$ represents the first epidemic peak during the testing phase. 
Further, it is possible to prove that $C\Xi$ decreases as $C$ increases for every sufficiently large $C$, where $\Xi=\Xi(C)$ is given by the solution of \eqref{xi_star_def}. From \eqref{xi_star_def}, we obtain
\[
\frac{d\Xi}{dC} = -\frac{r_{\max} x_{\I}(\Xi)}{C^2}.
\]
Since $d(C\Xi)/dC = \Xi + C d\Xi/dC$, therefore
\[
\frac{d(C\Xi)}{dC} = \Xi - \frac{r_{\max} x_{\I}(\Xi)}{C} = \int\limits_0^{T} \left[ x_\I(\tau) - x_\I(T) \right] d\tau
\]
where $\Xi$ is rewritten according to \eqref{eq:xi} and $r_{\max}/C$ is represented by the integral of 1 over the time interval $[0,T]$, where $T = r_{\max}/C$. This implies that $d(C\Xi)/dC<0$ unless $x_\I(T)$ is smaller than the average number of infected people in the time interval $[0,T]$. Thus, for every reasonable, sufficiently large $C$ such that $T$ is not too large, $x_\I^{p_1}$ decreases and $x_\I^{p_2}$ increases. 
This property can be used to minimize the maximum between two epidemic peaks. Indeed, the maximum between a decreasing and an increasing function is minimized when they are equal to each other. Thus, we equate two epidemic peaks in \eqref{x_I_two_peaks} and solve with respect to $\Xi$ in order to obtain the optimality condition
\begin{equation}\label{xi_optimality}
\Xi  = \frac{N}{\beta} \left( 1 + \ln R_1 - \frac{R_1}{R_2 - R_1} \ln \frac{R_2}{R_1} \right)
\end{equation}
which, along with \eqref{xi_star_def}, constitutes a complete system to determine the optimal $C$.


\begin{Proposition} \label{prop:COST}
Consider the testing rate $u(t)$ defined by \eqref{u-COST} and assume that $x_\T\approx(1-\theta)N$. Then, the COST policy $C$ is obtained by solving the system of equations \eqref{xi_star_def} and \eqref{xi_optimality} with $T:=r_{\max}/C$. This policy is unique provided that $R_1>1$.
\end{Proposition}

The proof of this proposition follows from the prior discussion. The solution $C$ for \eqref{xi_star_def} and \eqref{xi_optimality} can be obtained numerically \cite{niazi2020}.

\subsection{COST Policy for the COVID-19 Case of France}

Similar to the BEST policy, we evaluate the COST policy through its impact on the number of active ICU cases $B(t)$ and the cumulative number of deaths $E(t)$ after fitting the model outputs \eqref{activeICU_cases} and \eqref{total_dead_cases} with the data in the actual testing scenario. However, we fit these equations to the data by using different value of $\theta$ than obtained in the case of BEST policy. For the COST policy, the value of $\theta$ is obtained through model validation with the COVID-19 data of France \cite{niazi2020} under the assumption that $x_\T(t) \approx (1-\theta)N$.

\begin{figure}[!t]
    \centering
    \includegraphics[width=0.4\textwidth]{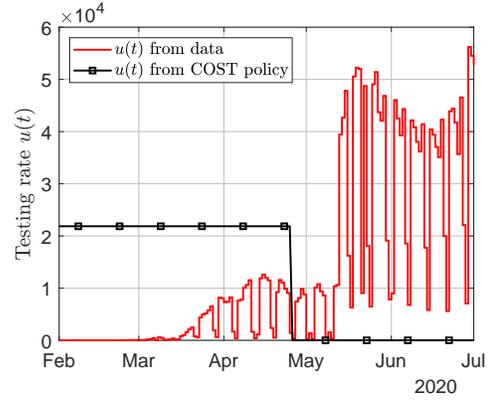}
    \caption{The comparison between the actual control input $u(t)$ versus the proposed control input $u(t)$ required by COST.}
    \label{fig:Input-COST}
\end{figure}

We consider that the total stockpile of tests $r_{\max}=2,038,037$, which is the total number of tests performed from January 24 to July 01, 2020, in France. Then, we obtain the optimal testing rate $C=17,144$ with the terminal time $T=118$ (i.e., May 20) from Proposition~\ref{prop:COST}. The actual testing rate of France and the testing rate required by the COST policy are illustrated in Figure~\ref{fig:Input-COST}. Notice that most of the tests in the actual testing are consumed after May 13, whereas all the tests in COST are allocated equally in the time interval from January 24 to May 20. The number of tests performed per day in the actual testing exceeds $50,000$ per day after May 13, which is about three times more than what is required by COST, i.e., $17,144$. Thus, COST is optimal in a sense that it is practical, it decreases the burden on medical facilities, and it reduces the number of deaths significantly. The impact of COST policy on the number of active ICU cases $B(t)$ and the cumulative number of deaths is illustrated in Figure~\ref{fig:ICU-COST} and Figure~\ref{fig:Death-COST}, respectively. From the figures, we observe that the COST policy reduces the peak of $B(t)$ by $11.12\%$ and the total number of cumulative deaths $E(t)$ by $37.52\%$. 
 
\begin{figure}[!t]
    \centering
    \includegraphics[width=0.4\textwidth]{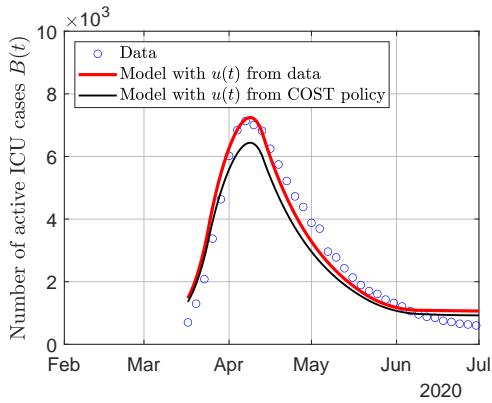}
    \caption{The comparison between the predicted number of active ICU cases $B(t)$ in the actual testing scenario and with COST policy.}
    \label{fig:ICU-COST}
\end{figure}
\begin{figure}[!t]
    \centering
    \includegraphics[width=0.4\textwidth]{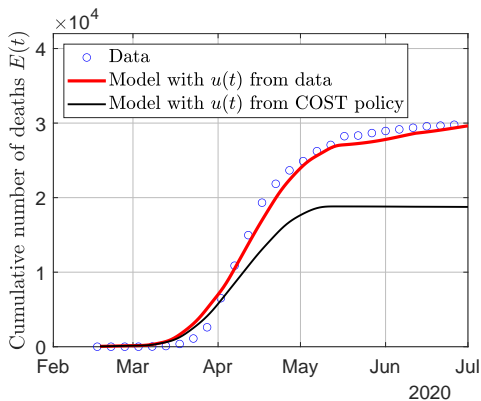}
    \caption{The comparison between the predicted cumulative number of deaths $E(t)$ in the actual testing scenario and with COST policy.}
    \label{fig:Death-COST}
\end{figure}

\section{Conclusion} \label{sec:conclusion}

We proposed a SIDUR model for demonstrating the impact of testing rate in the suppression and/or mitigation of an epidemic. The tests enable the health authority to diagnose and isolate the infected people, which limits the spread of disease to the susceptible population. Two testing policies have been proposed: 1)~best-effort strategy for testing (BEST) and 2)~constant optimal strategy for testing (COST). The BEST policy is a suppression policy that provides the minimum testing rate to be implemented from a certain day onward to stop the growth of undetected infected cases immediately. On the other hand, the COST policy is a mitigation policy that considers a limited stockpile of tests and optimally allocates the tests in a time interval so that the peak of the epidemic is minimized.
Although both BEST and COST policies are easy to compute and implement, they propose constant testing rates, which may limit the influence of testing if the epidemic parameters change due to external conditions. In the future, therefore, it will be interesting to devise a testing policy with a dynamic testing rate that minimizes the epidemic peak and/or the cumulative number of infected cases.



\begin{thebibliography}{10}
\providecommand{\url}[1]{#1}
\csname url@samestyle\endcsname
\providecommand{\newblock}{\relax}
\providecommand{\bibinfo}[2]{#2}
\providecommand{\BIBentrySTDinterwordspacing}{\spaceskip=0pt\relax}
\providecommand{\BIBentryALTinterwordstretchfactor}{4}
\providecommand{\BIBentryALTinterwordspacing}{\spaceskip=\fontdimen2\font plus
\BIBentryALTinterwordstretchfactor\fontdimen3\font minus
  \fontdimen4\font\relax}
\providecommand{\BIBforeignlanguage}[2]{{%
\expandafter\ifx\csname l@#1\endcsname\relax
\typeout{** WARNING: IEEEtran.bst: No hyphenation pattern has been}%
\typeout{** loaded for the language `#1'. Using the pattern for}%
\typeout{** the default language instead.}%
\else
\language=\csname l@#1\endcsname
\fi
#2}}
\providecommand{\BIBdecl}{\relax}
\BIBdecl

\bibitem{chowell2003}
G.~Chowell, P.~Fenimore, M.~Castillo-Garsow, and C.~Castillo-Chavez, ``{SARS}
  outbreaks in {Ontario}, {Hong Kong} and {Singapore}: The role of diagnosis
  and isolation as a control mechanism,'' \emph{Journal of Theoretical
  Biology}, vol. 224, no.~1, pp. 1--8, 2003.

\bibitem{nowzari2017}
C.~Nowzari, V.~M. Preciado, and G.~J. Pappas, ``Optimal resource allocation for
  control of networked epidemic models,'' \emph{IEEE Transactions on Control of
  Network Systems}, vol.~4, no.~2, pp. 159--169, 2017.

\bibitem{pezzutto2020}
M.~Pezzutto, N.~B. Rossello, L.~Schenato, and E.~Garone, ``Smart testing and
  selective quarantine for the control of epidemics,'' \emph{arXiv:2007.15412},
  2020.

\bibitem{ely2020}
J.~Ely, A.~Galeotti, and J.~Steiner, ``Optimal test allocation,'' Mimeo, Tech.
  Rep., 2020.

\bibitem{piguillem2020}
F.~Piguillem and L.~Shi, ``Optimal {COVID-19} quarantine and testing
  policies,'' \emph{CEPR Discussion Paper No. DP14613}, 2020.

\bibitem{berger2020}
D.~Berger, K.~Herkenhoff, and S.~Mongey, ``An {SEIR} infectious disease model
  with testing and conditional quarantine,'' \emph{{NBER} Working Paper No.
  26901}, 2020.

\bibitem{charpentier2020}
A.~Charpentier, R.~Elie, M.~Lauri\'{e}re, and V.~Tran, ``{COVID-19} pandemic
  control: {B}alancing detection policy and lockdown intervention under {ICU}
  sustainability,'' \emph{arXiv:2005.06526v3}, 2020.

\bibitem{giordano2020}
G.~Giordano, F.~Blanchini, R.~Bruno, P.~Colaneri, A.~{Di Filippo}, A.~{Di
  Matteo}, and M.~Colaneri, ``Modelling the {COVID-19} epidemic and
  implementation of population-wide interventions in {Italy},'' \emph{Nature
  medicine}, vol.~26, pp. 855--860, June 2020.

\bibitem{liu2020}
Z.~Liu, P.~Magal, O.~Seydi, and G.~Webb, ``Understanding unreported cases in
  the {COVID-19} epidemic outbreak in {Wuhan}, {China}, and the importance of
  major public health interventions,'' \emph{Biology}, vol.~9, no.~3, p.~50,
  2020.

\bibitem{ducrot2020}
A.~Ducrot, P.~Magal, T.~Nguyen, and G.~Webb, ``Identifying the number of
  unreported cases in {SIR} epidemic models,'' \emph{Mathematical medicine and
  biology: {A} journal of the {IMA}}, vol.~37, no.~2, pp. 243--261, 2020.

\bibitem{niazi2020}
M.~U.~B. Niazi, A.~Kibangou, C.~Canudas-de Wit, D.~Nikitin, L.~Tumash, and
  P.-A. Bliman, ``Modeling and control of {COVID-19} epidemic through testing
  policies,'' \emph{arXiv:2010.15438}, 2020.

\bibitem{hethcote2000}
H.~Hethcote, ``The mathematics of infectious diseases,'' \emph{SIAM Review},
  vol.~42, no.~1, pp. 599--653, 2000.

\bibitem{rothman2008}
K.~J. Rothman, S.~Greenland, and T.~L. Lash, \emph{Modern epidemiology},
  3rd~ed.\hskip 1em plus 0.5em minus 0.4em\relax Lippincott Williams \&
  Wilkins, 2008.

\bibitem{driessche2008}
P.~Van~den Driessche and J.~Watmough, ``Further notes on the basic reproduction
  number,'' in \emph{Mathematical epidemiology}.\hskip 1em plus 0.5em minus
  0.4em\relax Springer, 2008, pp. 159--178.

\end{thebibliography}
\end{document}